\documentclass[11pt]{amsart}

\usepackage{graphics}
\usepackage{graphicx}
\usepackage{epigraph} 
\usepackage{array}
\usepackage{enumitem} 
\usepackage{amsmath, amssymb, amsthm}
\usepackage{subfigure}
\usepackage{orcidlink}
\DeclareMathAlphabet{\mathpzc}{OT1}{pzc}{m}{it} 
\usepackage{mathrsfs}
\usepackage[small,bf]{caption}
\usepackage{tikz-cd}
\usepackage{geometry}
\makeatletter

\usepackage{hyperref} 

\usepackage[alphabetic]{amsrefs} 

\usepackage{tikz}
\usetikzlibrary{matrix,arrows}

\newtheorem{teorema}{Teorema}[section]

\theoremstyle{plain}

\newtheorem{teo}{Theorem}[section]
\newtheorem{definition}{Definition}

\newtheorem{lemma}[teorema]{Lemma}

\newtheorem{proposition}[teorema]{Proposition}

\numberwithin{equation}{section}

\newcommand{\Acal}{\mathcal{A}}

\newcommand{\Lcal}{\mathcal{L}}

\newcommand{\Ocal}{\mathcal{O}}

\newcommand{\Ical}{\mathcal{I}}
\newcommand{\Fcal}{\mathcal{F}}

\newcommand{\Pro}{\mathbb{P}}

\newcommand{\C}{\mathbb{C}}

\newcommand{\Q}{\mathbb{Q}}
\newcommand{\R}{\mathbb{R}}

\begin{document}

\title{An upper bound for the generalized greatest common divisor of rational points}
\author{Benjamín Barrios \orcidlink{0009-0002-8920-313X}}
\begin{abstract}
    Let $X$ be a smooth projective variety defined over a number field $K$. We give an upper bound for the generalized greatest common divisor of a point $x\in X$ with respect to an irreducible subvariety $Y\subseteq X$ also defined over $K$. To prove the result, we stablish a rather uniform Riemann--Roch type inequality.
\end{abstract}

\address{Departamento de Matemáticas, Pontificia Universidad Católica de Chile. Facultad de Matemáticas, 4860 Av. Vicuña Mackenna, Macul, RM, Chile.}
\email[B. Barrios]{benjamin.barriosco@gmail.com}%
\thanks{Supported by ANID Master's Fellowship Folio  22221062 from Chile.}
\date{\today}


\maketitle



\section{Introduction}

In the literature there are several bounds for the greatest common divisors in distinct cases. Bugeaud, Corvaja and Zannier \cite{MR1953049} gave an upper bound for the greatest common divisor of $a^n-1$ and $b^n-1$ where $a,b$ are multiplicatively independent integers. Later, Corvaja and Zannier \cite{MR2130274} generalized the previous result when $a^n$ and $b^n$ are replaced by elements of a fixed finitely generated subgroup of $\overline{\Q}^\ast$. In this way, Levin \cite{MR3910069} gave a GCD bound for polynomials in several variables with algebraic coefficients, generalizing the result obtained by Corvaja and Zannier. 

The GCD problem has been studied in various other settings. A similar result to the one by Bugeaud, Corvaja and Zannier was stated and proved by Ailon and Rudnick \cite{MR2046966} for non-constant and multiplicative independent complex polynomials. Some extensions of this last work were given by Ostafe in \cite{MR3539944}. In addition, a positive characteristic variation was established by Ghioca, Hsia and Tucker in \cite{MR3625452}.

In the context of Nevanlinna Theory, Pastén and Wang \cite{MR3632098} gave a GCD bound for algebraically independent meromorphic functions over $\C$. A stronger result of this type was obtained by Levin and Wang, see \cite{MR4160303}*{Theorem 1.3}.

We study the GCD problem in the context of varieties over number fields.

Let $K$ be a number field, $X$ be a smooth projective variety defined over $K$ and $\Lcal$ be a line sheaf on $X$. There is a height function $h_{X,\mathcal{L}}:X(\overline{K})\to\R$, see for instance \cite{MR883451}. For a divisor $D$ on $X$ we simply write $h_{X,D}$ instead of $h_{X,\Ocal(D)}$.

This height in the case of blow-ups is closely related to the generalized greatest common divisor. For the details see \cite{MR2162351}. 

\begin{definition}[\cite{MR2162351}*{Definition 2}]
    Let $X$ be a smooth variety defined over $K$ and $Y\subset X$ be an irreducible subvariety of $X$ also defined over $K$, of co-dimension $c\geq2$. Let $\pi:\Tilde{X}\to X$ be the blow up of $X$ along $Y$ and let $E_Y$ be the exceptional divisor of the blow up. For $x\in X\setminus Y$ we let $\tilde{x}=\pi^{-1}(x)$. \\
    The generalized (logarithmic) greatest common divisor of the point $x\in (X\setminus Y)(\overline{K})$ with respect to $Y$ is $$h_{\gcd}(x;Y)=h_{\Tilde{X},E_Y}(\tilde{x}).$$
\end{definition}

The goal of this work is to give an upper bound for $h_{gcd}(x;Y)$. For some examples of this quantity see \cite{MR2162351}.

Silvermann \cite{MR2162351} noticed that the generalized GCD is closely related to the Vojta's main conjecture. Bounds for the generalized GCD in varieties were given by Grieve in \cite{MR4121877} for rational points in certain toric varieties and by Wang and Yasukufu in \cite{MR4226990} under integrality conditions of the points. Recently, García-Fritz and Pastén \cite{MR4756363} gave an upper bound when the closed sub-scheme is reduced and consists on $d$ geometric points, without integrality conditions. In fact they have: 

\begin{teo}[\cite{MR4756363}*{Theorem 3.1}]
    Let $X$ be a smooth projective variety defined over a number field $K$ of dimension $n$ and $\Acal$ be an ample line sheaf on $X$. Let $Y$ be a reduced closed sub-scheme consisting on $d$ geometric points. Then given any $\varepsilon>0$ there is a properly contained Zariski closed set $Z_\varepsilon\subset X$ such that $$h_{\gcd}(x;Y)\leq \left(\sqrt[n]{\frac{d}{(\Acal^n)}}+\varepsilon\right)h(\Acal,x)+O(1)$$
    as $x$ varies in $(X-Z_{\varepsilon})(\overline{K})$. 
\end{teo}

In this work we follow the ideas of García-Fritz and Pastén to give an upper bound for the GCD when the closed sub-scheme is a higher dimensional irreducible sub-variety. Explicitly we obtain:

\begin{teo}[GCD bound]\label{thm2}
    Let $X$ be a smooth projective variety defined over a number field $K$ of dimension $n$ and $Y$ be an irreducible sub-variety of dimension $d$, also defined over $K$. Let $X'=X\setminus Y$. Then given any ample line sheaf $\mathcal{A}$ and $\varepsilon>0$, there is a properly contained Zariski closed set $Z_\varepsilon\subset X'$ such that, for all $x\in (X'\setminus Z_\varepsilon)(\overline{K})$ the formula: 
    $$h_{\gcd}(x;Y)\leq \left(\left(\frac{(\mathcal{A}^d\cdot Y)}{(\mathcal{A}^n)}\cdot\frac{n!}{c!}\right)^{\frac{1}{c}}+\varepsilon\right)h(\mathcal{A},x)+O_{\varepsilon}(1)$$
    holds, where $c=n-d$.
\end{teo}

The 0-dimensional case studied by Pastén and García-Fritz has interesting applications in terms of the Bombieri--Lang conjecture and the Vojta conjecture, see Theorems 1.4 and 3.2 in \cite{MR4756363}. We expect similar consequences from Theorem \ref{thm2}, which we leave for future research. 

The proof of Theorem \ref{thm2} is inspired by the methods of \cite{MR4756363}. However, new complications appear and we need to prove an auxiliary result on the Asymptotic Riemann--Roch Theorem, see Theorem \ref{nnrr}.

\section{Riemann--Roch}

In this section we will give an estimate of the dimension of certain space that will be useful for the GCD bound.

We recall the Asymptotic Riemann--Roch theorem:

\begin{teo}(Asymptotic Riemann--Roch)\label{rr}
    Let $X$ be a projective variety of dimension $n$, $\Lcal$ be an ample line sheaf on $X$ and $\Fcal$ be a coherent sheaf on $X$. Then \begin{align*}
        h^i(X,\Lcal^{\otimes m}\otimes \Fcal)&=O(m^{n-1})~\mathrm{for}~i>0~\mathrm{and} \\
        h^0(X,\Lcal^{\otimes m}\otimes \Fcal)&=\frac{(\Lcal^n\cdot \Fcal)}{n!}m^n+O(m^{n-1})
    \end{align*}
    where the constant in $O(m^{n-1})$ depends on $X,\Fcal$ and $\Lcal$.
\end{teo}

Suppose that $X$ is a projective variety over a field $K$ and $Y\subset X$ is a subvariety of $X$ also defined over $K$. Then we have the ideal sheaf $\Ical_Y$ associated to $Y$, and for $r\geq 1$ we consider the coherent sheaf $\Ocal_X/\Ical_Y^r$. In Section \ref{sec3} we will need to estimate $h^0(X,\Lcal^{\otimes m}\otimes \Ocal_X/\Ical_Y^r)$ but we will have to know how this dimension changes as $r$ increases, which Theorem \ref{rr} does not give us. 

Nevertheless, the following two results will help us in to proceed.

\begin{lemma}\label{teo1.3.2}
    Let $Z$ be a projective and of finite type scheme over a number field $K$ of dimension $d$. Let $\mathcal{A}$ be an ample base point free line sheaf on $Z$. Then $$h^0(Z,\mathcal{A})\leq (\mathcal{A}^d)+d.$$
\end{lemma}
\begin{proof}
    Let $q:Z\to\Pro^{n}$ be the morphism defined by $\mathcal{A}$ and let $Y=q(Z)$ be its image. Then $h^0(Z,\Acal)=n+1$. Let $n_Y$ be the dimension of $Y$. \\
    First, we wil proof that $n+1\leq (\Ocal_{\Pro^n}(1)^{n_Y}\cdot Y)+n_Y$. We will do induction on $n_Y$. 
    
    If $n_Y=0$ then the inequality holds. 
    
    If $n_Y\geq 1$, let $H\subset\Pro^n$ be a hyperplane meeting $Y$ properly. Consider $Y'=Y\cap H\subset \Pro^{n-1}$. Then there is a properly contained $Z'\subset Z$ and a morphism $Z'\to Y'\subset \Pro^n$. By \cite{MR463157}*{Theorem 7.1, II} this morphism comes from an ample and base point free line sheaf $\Acal'$ on $Z'$. Then, by induction $$(n-1)+1\leq (\Ocal_{\Pro^{n-1}}(1)^{n_Y-1}\cdot Y')+(n_Y-1) \Rightarrow n+1 \leq (\Ocal_{\Pro^{n-1}}(1)^{n_Y-1}\cdot Y')+n_Y.$$
    Furthermore $(\Ocal_{\Pro^{n-1}}(1)\cdot Y')=(\Ocal_{\Pro^{n-1}}\cdot H\cap Y)=(\Ocal_{\Pro^n}(1)\cdot Y)$. 
    
    Therefore $$n+1\leq (\Ocal_{\Pro^n}(1)\cdot Y)+n_Y.$$
    Finally, by \cite{MR2095471}*{Corollary 1.2.15} the morphism $q:Z\to Y$ is finite. Then $q_{\ast}Z=\deg(q)\cdot Y$ in $K_{d}(\Pro^n)$. By Projection formula $(A^d\cdot Z)=t(\Ocal_{\Pro^n}(1)\cdot Y)$. Since $n_y\leq d$, we conclude.
\end{proof} 

\begin{teo}(Riemann--Roch type inequality)\label{nnrr}
    Let $X$ be a projective variety defined over $K$ and $Y\subset X$ be an irreducible subvariety of dimension $d$. Let $\Ical_Y$ be the ideal sheaf associated to $Y$. Then given any very ample line sheaf $\Acal$ on $X$, we have: $$h^0(X,\mathcal{A}^{\otimes m}\otimes \Ocal_X/\Ical_Y^r)\leq \frac{r^c\cdot m^{d}}{c!}e_X(Y)\cdot(\mathcal{A}^{d}\cdot Y)+O(r^{c-1})$$
    for $r\gg1$, where $c=\dim X-d$, $e_{Y}(X)$ is the algebraic multiplicity of $X$ along $Y$ and the constant in $O(r^{c-1})$ only depends on $X,Y$ and $\mathcal{A}$.
\end{teo}
\begin{proof}
    Since $\Acal$ is very ample, then $\Acal|_{Y_r}$ is very ample on $Y_r$. By Lemma  then \begin{align}
        h^0(X,\Acal^{\otimes m}\otimes \Ocal_X/\Ical_Y^r)\leq (\Acal^d\cdot \Ocal_X/\Ical_Y^r)m^d+d.
    \end{align}
    Let $\xi_Y$ be the generic point of $Y$, then $$(\Acal^d\cdot\Ocal_X/\Ical_Y^r)=\mathrm{length}_{\xi_Y}\left(\Ocal_X/\Ical_Y^r\right)_{\xi_Y}(\Acal^d\cdot Y)$$ 
    By \cite{MR1644323}*{Example 4.3.4} this length is equal to $$e_Y(X)\cdot\frac{r^c}{c!}+O(r^{c-1})$$
    for $r\gg 1$. Replacing this in (2.1) give the result.
\end{proof}

\section{Proof of the GCD bound}\label{sec3}

In this section we obtain a bound for the generalized greatest common divisor. 

Troughout this section: \begin{itemize}
    \item $K$ is a number field 
    \item $X$ is a smooth projective variety defined over $K$ 
    \item $Y$ is an irreducible subvariety of $X$ of codimension $c\geq 2$ also defined over $K$
    \item $X'=X\setminus Y$
\end{itemize}

\begin{proposition}\label{prop2.3.1}
    For all ample line sheaf $\mathcal{A}$ on $X$ if there is an effective divisor $D$ on $X$ such that $\mathcal{O}(D)=\mathcal{A}^{\otimes m}$ and $r=m_Y(D)$ for some positive integers $m,r$, there is a properly contained Zariski closed set $Z\subset X$ such that for all $x\in (X'\setminus Z)(\overline{K})$: $$h_{\gcd}(x;Z)\leq \frac{m}{r}h(\mathcal{A},x)+O(1)$$
\end{proposition}
\begin{proof}
    Let $E_Y$ be the exceptional divisor of the blow-up of $X$ along $Y$ and $B$ be the base locus of $\tilde{D}$. By the assumption of $D$, given any $y\in(\tilde{X}\setminus B)(\overline{K})$, let $\pi(x)=y$, then we have \begin{align*}
    r\cdot h_{\gcd}(x;Y) &= r\cdot h_{\tilde{X},E_Y}(y)+O(1) \\
    &\leq h_{\tilde{X},\pi^{\ast}(D)}(y)+O(1) \\
    &= h_{X,D}(x)+O(1)\\
    &= m\cdot h_{X,\mathcal{A}}(x)+O(1)
    \end{align*}
    Thus $$h_{\gcd}(x;Y)\leq \frac{m}{r}h_{X,\mathcal{A}}(x)+O(1)$$
    for all $x\in (X'\setminus Z)(\overline{K})$ with $Z=\pi(B)$.
\end{proof}

\begin{teo}\label{teo2.3.5}
    Given any ample line sheaf $\mathcal{A}$ on $X$ and $\varepsilon>0$ there is a properly contained Zariski closed set $Z_\varepsilon\subset X$ such that for all $x\in (X'\setminus Z_\varepsilon)(\overline{K})$: $$h_{\gcd}(x;Y)\leq \left(\left(\frac{(\mathcal{A}^d\cdot Y)}{(\mathcal{A}^n)}\cdot\frac{n!}{c!}\right)^{\frac{1}{c}}+\varepsilon\right)h_{X,\mathcal{A}}(x)+O_{\varepsilon}(1).$$
\end{teo}
\begin{proof}
    First consider the case when $\Acal$ is very ample. By Proposition \ref{prop2.3.1} we have to guarantee the existence of a divisor $D$ in $X$ such that $\Ocal(D)=\Acal^{\otimes m}$ and $r=m_Y(D)$ for some $m,r$. 

    Let $m$ and $r$ be positive integers and $\Ical_Y$ be the ideal sheaf associated to $Y$, then the associated exact sequence to $\Ical_Y^r$ is \begin{equation*}
        0\to \Ical_Y^r\to \mathcal{O}_X\to \mathcal{O}_X/\Ical_Y^r\to 0
    \end{equation*}
    tensoring by $\Acal^{\otimes m}$ and passing to global sections we obtain \begin{equation*}
        0\to H^0(X,\mathcal{A}^{\otimes m}\otimes \mathcal{I}_Y^r)\to H^0(X,\mathcal{A}^{\otimes m})\to H^0(X,\mathcal{A}^{\otimes m}\otimes \mathcal{O}_X/\mathcal{I}_Y^r)
    \end{equation*}
    If $h^0(X,\mathcal{A}^{\otimes m}) > h^0(X,\mathcal{A}^{\otimes m}\otimes \mathcal{O}_X/\mathcal{I}_Y^r)$ then $H^0(X,\mathcal{A}^{\otimes m}\otimes \mathcal{I}_Y^r)$ is non-trivial and we obtain the desired divisor (in fact we obtain $m_Y(D)\geq r$ but this is enough).

    Now, given any $\varepsilon>0$ there is a sufficiently large choice for $m,r$ such that $$\frac{(\Acal^n)}{n!}m^n > \frac{r^c\cdot m^d}{c!}(\Acal^d\cdot Y)~~~~\mathrm{and}~~~~\frac{m}{r}<\left(\frac{(\Acal^d\cdot Y)}{(\Acal^n)}\cdot\frac{n!}{c!}\right)^{\frac{1}{c}}+\varepsilon$$
    By Theorems \ref{rr} and \ref{nnrr}, when $\Acal$ is very ample we obtain the result.

    Finally, if $\Acal$ is ample consider an integer $k$ such that $\mathcal{A}^{\otimes k}$ is very ample. Applying the result to $\mathcal{A}^{\otimes k}$ we obtain that \begin{align*}
        h_{\gcd}(x;Y)&\leq \left(\left(\frac{((\mathcal{A}^{\otimes k})^d\cdot Y)}{((\mathcal{A}^{\otimes k})^n)}\cdot\frac{n!}{c!}\right)^{\frac{1}{c}}+\varepsilon\right)h_{X,\mathcal{A}^{\otimes k}}(x)+O_{\varepsilon}(1) \\
        &\leq \left(\left(\frac{k^d(\mathcal{A}^d\cdot Y)}{k^n(\mathcal{A}^n)}\cdot\frac{n!}{c!}\right)^{\frac{1}{c}}+\varepsilon\right)k\cdot h_{X,\mathcal{A}}(x)+O_{\varepsilon}(1) \\
        &\leq \left(\frac{1}{k}\left(\frac{(\mathcal{A}^d\cdot Y)}{(\mathcal{A}^n)}\cdot\frac{n!}{c!}\right)^{\frac{1}{c}}+\varepsilon\right)k\cdot h_{X,\mathcal{A}}(x)+O_{\varepsilon}(1).
    \end{align*}
    From this, we can conclude.
\end{proof}




\section{Acknowledgments}

The author was supported by ANID Master's Fellowship Folio 22221062.

This work arise in the context of my Master's thesis, because of this I would like to thank to my advisor Héctor Pastén and to the readers Siddarth Mathur and Ricardo Menares for their comments. I also thank to the referee for his suggestions to improve the content of this article.


\end{document}